\documentclass[12pt]{amsart}

\usepackage{cite}
\usepackage{amsfonts}
\usepackage{amscd}
\usepackage{amsthm}
\usepackage[all]{xy}
\usepackage[dvips]{graphicx}
\usepackage{graphicx}
\usepackage{verbatim}
\usepackage{enumerate}
\usepackage{amsmath}
\usepackage{amssymb}
\usepackage{url}
\usepackage{xcolor}
\usepackage{pgf,tikz}
\usepackage{mathrsfs}
\usetikzlibrary{arrows}

\definecolor{qqwuqq}{rgb}{0,0.39215686274509803,0}
\definecolor{qqqqff}{rgb}{0,0,1}
\definecolor{ffqqqq}{rgb}{1,0,0}

\newcommand{\DD}{\mathbb{D}}
\newcommand{\RR}{\mathbb{R}}

\newcommand{\ZZ}{\mathbb{Z}}
\newcommand{\TT}{\mathbb{T}}

\newcommand{\ve}{\varepsilon}

\newcommand{\CCC}{\mathcal{C}}

\newcommand{\CM}{\mathcal{M}}
\newcommand{\CN}{\mathcal{N}}
\newcommand{\CK}{\mathcal{K}}

\newcommand{\diam}{\mathrm{diam}}
\newcommand{\dist}{\mathrm{dist}}

\newcommand{\wH}{\widetilde{H}}

\newcommand{\p}{\partial}

\newcommand{\wpsi}{\widetilde{\psi}}

\newcommand{\id}{\mathrm{id}}
\newcommand{\inter}{\mathrm{int}}
\newcommand{\Leb}{\mathrm{Leb}}

\newcommand{\oDD}{\overline{\DD}}

\newcommand{\hF}{\widehat{F}}

\newcommand{\hvp}{\widehat{\varphi}}
\newcommand{\hpsi}{\widehat{\psi}}
\newcommand{\hrho}{\widehat{\rho}}
\newcommand{\hf}{\widehat{f}}

\newcommand{\hH}{\widehat{H}}
\newcommand{\hX}{\widehat{X}}
\newcommand{\homega}{\widehat{\omega}}

\newtheorem{theorem}{Theorem}[section]
\newtheorem{proposition}[theorem]{Proposition}
\newtheorem{lemma}[theorem]{Lemma}
\newtheorem{sublemma}[theorem]{Sublemma}

\newtheorem{remark}[theorem]{Remark}

\numberwithin{equation}{section}

\newtheorem*{theoremA}{Theorem A}
\newtheorem*{theoremB}{Theorem B}
\newtheorem*{theoremC}{Theorem C}

\def\a{\alpha}
\def\b{\beta}
\def\vvphi{\rho}

\title[The essential coexistence phenomenon]
{The essential coexistence phenomenon in Hamiltonian dynamics}

\author{Jianyu Chen}
\address{Department of Math. \& Stat., University of Massachusetts Amherst, Amherst, MA 01003, USA}
\email{jchen@math.umass.edu}
\author{Huyi Hu}
\address{Department of Mathematics, Michigan State University, East Lansing, MI 48824, USA}
\email{hhu@math.msu.edu}
\author{Yakov Pesin}
\address{Department of Mathematics \\ Pennsylvania State University University Park, PA 16802, USA}
\email{pesin@math.psu.edu}
\author{Ke Zhang}
\address{Department of Mathematics, University of Toronto, Toronto, Ontario, Canada}
\email{kzhang@math.toronto.edu}

\begin{document}

\date{\today}

\begin{abstract}
We construct an example of a Hamiltonian flow $f^t$ on a $4$-dimensional smooth manifold $\mathcal{M}$ which after being restricted to an energy surface 
$\mathcal{M}_e$ demonstrates essential coexistence of regular and chaotic dynamics that is there is an open and dense $f^t$-invariant subset $U\subset\mathcal{M}_e$ such that the restriction $f^t|U$ has non-zero Lyapunov exponents in all directions (except the direction of the flow) and is a Bernoulli flow while on the boundary $\p U$, which has positive volume all Lyapunov exponents of the system are zero. 
\end{abstract}

\thanks{Ya.~P. was partially supported by NSF grant DMS-1400027.}

\maketitle

\section{Introduction}\label{SIntroduction}

The problem of essential coexistence of regular and chaotic behavior lies in the core of the theory of smooth dynamical systems. The early development of the theory of dynamical systems was focused on the study of regular behavior in dynamics such as presence and stability of periodic motions, translations on surfaces, etc. The first examples of systems with highly complicated ``chaotic'' behavior -- so-called homoclinic tangles -- were already discovered by Poincar\'e in 1889 in conjunction with his work on the three-body problem. However, a rigorous study of chaotic behavior in smooth dynamical systems began in the second part of the last century due to the pioneering work of Anosov, Sinai, Smale and others which has led to the development of hyperbolicity theory. It was therefore natural to ask whether the two types of dynamical behavior -- regular and chaotic -- can coexist in an essential way. 

While regular and chaotic dynamics can coexist in various ways, in this paper we will consider one of the most interesting situations that can be described as follows. Let $f^t$ be a volume preserving smooth dynamical system with discrete or continuous time $t$ acting on a compact smooth manifold $M$. We say that $f^t$ exhibits essential coexistence if there is an open and dense $f^t$-invariant subset $U\subset M$ such that the restriction $f^t|U$ has non-zero Lyapunov exponents in all directions (except the direction of the flow for continuous $t$) and is a Bernoulli system while on the boundary 
$\p U$, which has {\bf positive volume}, all Lyapunov exponents of the system are zero. Note that it is the requirement that the boundary of $U$ has positive volume that makes coexistence essential. It follows that the Kolmogorov-Sinai entropy of $f^t|U$ is positive while the topological entropy of $f^t|\p U$ is zero, see the survey \cite{CHP13} for more information on the essential coexistence, some results, examples and open problems.  

The study of essential coexistence was inspired by discovery of KAM phenomenon in Hamiltonian dynamics and a similar phenomenon in the space of volume preserving systems. The latter was shown in the work of Cheng--Sun \cite{CS90}, Herman \cite{Her90}, Xia \cite{Xia92}, and Yoccoz \cite{Yoc92} who proved that on any manifold $M$ and for any sufficiently large $r$ there are open sets of volume preserving $C^r$ diffeomorphisms of $M$ all of which possess positive volume sets of co-dimension-$1$ invariant tori; on each such torus the diffeomorphism is $C^1$ conjugate to a Diophantine translation; all of the Lyapunov exponents are zero on the invariant tori. The set of invariant tori is nowhere dense and it is expected that it is surrounded by ``chaotic sea'', i.e., outside this set the Lyapunov exponents are nonzero almost everywhere and hence, the system has at most countably many ergodic components. It has since been an open problem to find out to what extend this picture is true.

First examples of systems with discrete and continuous time demonstrating essential coexistence, which are volume preserving, were constructed in \cite{HPT13, CHP13a, C12} (see also \cite{CHP13} for a survey of recent results). Naturally, one would like to construct examples of systems with essential coexistence which are Hamiltonian. This is what we do in this paper: we present an example of a Hamiltonian flow on a $4$-dimensional manifold which demonstrates the essential coexistence phenomenon, see Theorem A. In the course of our construction we also obtain an area preserving  $C^\infty$ diffeomorphism of a $2$-dimensional torus as well as a volume preserving $C^\infty$ flow on a $3$-dimensional manifold both demonstrating the essential coexistence phenomenon, see Theorems C and B respectively. These examples are simpler than the ones in \cite{HPT13} and \cite{CHP13a}, where the corresponding constructions require $5$-dimensional manifolds.

In Section~\ref{SStatement} we give a brief introduction to Hamiltonian dynamics recalling some basic facts and we state our main results, Theorems A, B, and C, that lead to the desired example of a Hamiltonian flow with essential coexistence. In the following three sections we present the proofs of these results. 

In particular, in Section 2 we construct a specific area preserving embedding of the closed unit disk in $\mathbb{R}^2$ onto a $2$-dimensional torus, which maps the interior of the disk onto an open and dense subset of the torus whose complement has positive area. 

This technically involved geometric construction lies at the heart of our proof of Theorem C which is presented, along with the proof of Theorem B, in Section 3. Another important ingredient of the proofs of these two theorems is the map of the $2$-dimensional unit disk known as the Katok map. This is a $C^\infty$ area preserving diffeomorphism, which is identity on the boundary of the disk and is arbitrary flat near the boundary. It is ergodic and in fact, is Bernoulli, and has non-zero Lyapunov exponents almost everywhere. It was introduced by Katok in \cite{Katok79} as the first basic step in his construction of $C^\infty$ area preserving Bernoulli diffeomorphisms with non-zero Lyapunov exponents on any surface. We will also use in a crucial way a result from \cite{HPT04} showing that the Katok map is smoothly isotopic to the identity map of the disk. This isotopy allows us to construct a flow on the $3$-dimensional torus with essential coexistence. Finally, in Section 4 we give the proof of Theorem A which provides the desired Hamiltonian flow with essential coexistence.

We note that while the Hamiltonian flow we construct in this paper displays the essential coexistence phenomenon, it does not quite exhibit KAM phenomenon, since the Cantor set of invariant tori are just unions of circles, on all of which the flow has the same linear speed. One can modify the construction in particular, increasing the dimension of invariant tori, to obtain a flow with Diophantine velocity vector on the tori.

{\bf Acknowledgement} The authors would like to thank Banff International Research Station, where part of the work was done, for their hospitality.

\section{Statement of Results}\label{SStatement}

A standard reference to Hamiltonian systems is  \cite{Arnold2006}. 
A \emph{symplectic manifold} is a smooth manifold $\CM$ equipped with a \emph{symplectic form} that is a closed non-degenerate $2$-form $\omega$. Necessarily, $\CM$ must be even dimensional (say $2n$) and $\omega^n$ defines a volume form on $\CM$. In the particular case $\CM = T^*N$, the cotangent bundle of a smooth manifold $N$, there is a natural symplectic form 
$\omega=dp\wedge dq :=\sum_{i = 1}^{n}dp_i\wedge dq_i$, where $(q, p)$ are the local coordinate in $T^*N$ induced by the local coordinate $q$ in $N$. In particular, the associated volume form is the standard one. 

Let $H$ be a $C^2$ function on a symplectic manifold $\CM$ called a \emph{Hamiltonian function}. The (autonomous) \emph{Hamiltonian vector field} $X_H$ is the unique vector field such that $\omega(X_H, \cdot) = dH(\cdot)$. For the symplectic form 
$\omega=dp\wedge dq$ one has $X_H=(\partial_p H(q, p), - \partial_q H(q, p))$. Let $f_H^t$ denote the \emph{Hamiltonian flow} generated by $X_H$. One can show that $f_H^t$ preserves the symplectic form, the volume form $\omega^n$, and the Hamiltonian function $H$. As a result, each level surface $\CM_e=\{H = e\}$, called an \emph{energy surface}, is invariant under the flow. As such, ergodic properties of an (autonomous) Hamiltonian flow are often discussed by restricting the flow to an energy surface. 

A flow $f^t$ on $\CM$ is \emph{ergodic} or \emph{Bernoulli} if for any $t\not= 0$, the time $t$ map of $f^t$ is ergodic or Bernoulli respectively. A flow $f^t$ is \emph{hyperbolic} if the flow has nonzero Lyapunov exponents almost everywhere, except for the flow direction. We say that an orbit of a flow has zero Lyapunov exponents if  the flow has zero Lyapunov exponents at any point of the orbit. Denote by $m$ the Lebesgue measure on $\CM$.

\begin{theoremA}\label{ThmA} 
There exists a $4$-dimensional manifold $\CM$ and a Hamiltonian function 
$H: \CM\to \RR$ such that restricted to any energy surface $\CM_e:=\{H=e\}$, the Hamiltonian flow $f^t: \CM_e\to \CM_e$ demonstrate the essential coexistence phenomenon that is 
\begin{enumerate}
\item[(a)] there is an open and dense set $U=U_e\in\CM_e$ such that $f^t(U)=U$ for any $t\in \RR$ and $m(U^c)>0$, where $U^c=\CM_e\setminus U$ is the complement of $U$;
\item[(b)] restricted to $U$, $f^t|U$ is hyperbolic and ergodic; in fact, $f^t|U$ is Bernoulli;
\item[(c)] restricted to $U^c$, all orbits of $f^t|{U^c}$ are periodic with zero Lyapunov exponents.
\end{enumerate}
\end{theoremA}

\begin{remark}
Actually our construction guarantees the following property of the flow $f^t|U^c$. For any $x\in U^c$, consider a small surface $\Sigma$ through $x$ which is transversal to the flow direction and let $\tilde f$ the corresponding Poincar\'e map. Then we have 
$\tilde{f}|\Sigma=\id$ and $D^k\tilde f|T_x\Sigma=0$ for any $k>0$, where $T_x\Sigma$ is the tangent space to $\Sigma$ at $x$.
\end{remark}

We obtain a flow in Theorem A by constructing a $C^\infty$ volume preserving flow on 
$\TT^3$ with essential coexistence.

\begin{theoremB}\label{ThmB} 
There exists a volume preserving $C^\infty$ flow $f^t$ on $\CM=\TT^3$ that demonstrates the essential coexistence phenomenon, i.e., it has Property (a)-(c) in Theorem A.
\end{theoremB}

In \cite{CHP13a}, the first three authors of this paper constructed a volume preserving $C^\infty$ flow $f^t$ on a $5$-dimensional manifold with essential coexistence. Moreover, in that paper, $U^c$ is a union of $3$-dimensional invariant submanifolds and $f^t$ is a linear flow with a Diophantine frequency vector on each invariant submanifold. In the example given by Theorem A, $U^c$ is a union of one-dimensional closed 
orbits since the center direction is only one-dimensional.

The proof that Theorem B implies Theorem A is given in Section~\ref{SThm A}. 

\smallskip
To obtain the flow in Theorem B, we construct a $C^\infty$ area preserving diffeomorphism on $\TT^2$ that demonstrates the essential coexistence phenomenon.

\begin{theoremC}\label{ThmC} 
There exists a $C^\infty$ area preserving diffeomorphism $f$ on $\TT^2$ such that \begin{enumerate}
\item[(a)] there is an open and dense set $U$ such that $f(U)=U$ and $m(U^c)>0$, where $U^c=\TT^2\setminus U$ is the complement of $U$;
\item[(b)] restricted to $U$, $f|U$ is hyperbolic and ergodic; in fact, $f|U$ is Bernoulli;
\item[(c)] restricted to $U^c$, $f|{U^c}=\id$ (and hence, its Lyapunov exponents are all zero).
\end{enumerate}
\end{theoremC}
In \cite{HPT13}, the authors constructed a $C^\infty$ volume preserving 
diffeomorphism $h$ of a five dimensional manifold that also has Properties (a)--(c), where $U^c$ is a union of $3$-dimensional invariant submanifolds and restricted to $U^c$, $h$ is the identity map (hence, with zero Lyapunov exponents).

Also, in \cite{C12}, the first author constructed a $C^\infty$ volume preserving diffeomorphism $h$ of a $4$-dimensional manifold with the same properties but the chaotic part has countably many ergodic components. There are other examples of dynamical systems that exhibit coexistence of chaotic and regular behavior though the regular part may not form a nowhere dense set, see the related reference in \cite{CHP13}.)

\section{Embedding the unit $2$-disk into the $2$-torus}\label{SProp}

In this section we state our main technical result, which provides a $C^\infty$ embedding from the open unit $2$-disk $\DD^2$ into the $2$-torus $\TT^2$ such that the image is open and dense but not of the full Lebesgue measure. We equip $\DD^2$ and $\TT^2$ with the standard Euclidean metric induced from $\RR^2$ and we denote by $d=d(x,y)$ the standard distance.

Let $M$ and $N$ be manifolds and $h: M\to N$ a $C^\infty$ diffeomorphism. The map $h$ induces a map $h^*: \bigwedge^2(N)\to \bigwedge^2(M)$, where $\bigwedge^2$ denotes the set of $2$-forms. Since $2$-forms are volume (area) forms, we can regard that $h^*$ sends smooth measures on $N$ to that on $M$. By slightly abusing notations we identity a $2$-form with the smooth measure given by this form  and with the density function of the smooth measure.

Let $m_{\DD^2}$ denote the normalized Lebesgue measure on ${\DD^2}$.

\begin{proposition}\label{PropMain} 
There exists a $C^\infty$ diffeomorphism $h$ from $\DD^2$ into $\TT^2$ with the following properties:
\begin{enumerate}
\item the image $U=h(\DD^2)$ is an open dense simply connected subset of $\TT^2$; 
moreover, $\p U=\TT^2\backslash U=E\cup L$, where $E$ is a Cantor set of positive Lebesgue measure and $L$ is a union of countably many line segments;
\item $h^*m_U=m_{\DD^2}$, where $m_U$ is the normalized Lebesgue measure on $U$;
\item $h$ can be continuously extended to $\p\DD^2$ such that $h(\p\DD^2)=\p U$,
and therefore for any $\ve>0$, $\CN_\ve=h^{-1}(V_\ve)$ is a neighborhood of 
$\p \DD^2$, where $V_\ve:=\{x\in U: d(x, \p U)<\ve\}$.
\end{enumerate}
\end{proposition}
In the proof of the proposition, we make an explicit construction of the map $h$. Note that $h$ can be extended continuously to the boundary $\p\DD^2$, and since 
$\dim_H(\p\DD^2)=1<2=\dim_H(\p U)$ (here $\dim_H$ denotes the Hausdorff dimension) the map cannot be Lipschitz on $\p\DD^2$.

We remark that one can use the Riemann mapping theorem to directly obtain a conformal $C^\infty$ diffeomorphism $h$ satisfying Statement 1 of Proposition~\ref{PropMain}. However, due to Carath\'eodory's theorem a conformal map can be extended to a homeomorphism of the closure of $U$ if and only if $\p U$ is a Jordan curve, which is not the case here. Therefore, such a map does not satisfy Statement 3 of Proposition~\ref{PropMain}, which is crucial for our construction.

In the proof below, we call a set a \emph{cross} of size $(\a, \b)$, where $\a<\b$, if it is 
the image under a translation of the set
\begin{equation*}
\Big(-\frac{\b}{2}, \ \frac{\b}{2}\Big)
\times \Big(-\frac{\a}{2}, \ \frac{\a}{2}\Big)
\bigcup
\Big(-\frac{\a}{2}, \ \frac{\a}{2}\Big)
\times \Big(-\frac{\b}{2}, \ \frac{\b}{2}\Big).
\end{equation*}
Roughly speaking, an $(\a, \b)$-cross consists of two intersecting open rectangles of size $\a\times \b$. The \emph{left} or \emph{right edge} of an $(\a, \b)$-cross is the image of the interval
\begin{equation*}
I_l=I_{l,\a, \b}= \Big\{-\frac{\b}{2}\Big\} 
\times \Big(-\frac{\a}{2}, \ \frac{\a}{2}\Big)
\quad \text{or} \quad 
I_r=I_{r,\a, \b}= \Big\{\frac{\b}{2}\Big\} 
\times \Big(-\frac{\a}{2}, \ \frac{\a}{2}\Big)
\end{equation*}
under the same translation respectively. The \emph{top} or \emph{bottom edge} of a cross is understood similarly.

We say that an $(\a, \b)$-cross is inscribed in a square of size $\b$
if the cross is contained in the square.

\begin{proof}[Proof of Proposition~\ref{PropMain}] {{}}
We split the proof into several steps.

{\bf Step 1.} We give an explicit construction of the sets $U$, $E$ and $L$. Consider the $2$-torus $\TT^2$ which we regard as $[0, 1]^2$ with opposite sides identified. Fix $\alpha\in(0, 0.05)$. We shall inductively construct a sequence of triples of disjoint sets $(U_n, E_n, L_n)$, where
\begin{itemize}
\item $U_n$ is a simply connected open subset in $\TT^2$ and $U_n\supset U_{n-1}$;
\item $E_n$ is a disjoint union of $4^n$ identical closed squares and 
$E_n\subset E_{n-1}$;
\item $L_n=\TT^2\backslash (U_n\cup E_n)$ consists of finitely many line segments and $L_n\supset L_{n-1}$.
\end{itemize}
Let $U_1$ be a $(\a, 1)$ cross inscribed in $\TT^2=[0, 1]^2$, $E_1$ the union of four closed squares of size $\beta_1=\frac{1-\alpha}{2}$ in the complement of $U_1$, and $L_1$ consists of four line segments, which are the left/right edges and top/bottom edges of the cross $U_1$. The four squares in $E_1$ are pairwise disjoint except on the boundary of $[0, 1]^2$, see Figure~\ref{fig: U12}. 

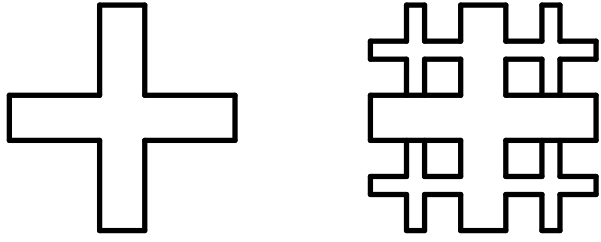
\begin{figure}[hpt]
\centering
\begin{tikzpicture}[line cap=round,line join=round,>=triangle 45,x=.3cm,y=.3cm]
\clip(-18,-5.5) rectangle (13,6.5);
\draw [line width=2pt] (0,2)-- (4,2);
\draw [line width=2pt] (6,6)-- (4,6);
\draw [line width=2pt] (6,2)-- (10,2);
\draw [line width=2pt] (10,2)-- (10,0);
\draw [line width=2pt] (10,0)-- (6,0);
\draw [line width=2pt] (6,-4)-- (4,-4);
\draw [line width=2pt] (4,0)-- (0,0);
\draw [line width=2pt] (0,0)-- (0,2);
\draw [line width=2pt] (4,2)-- (4,3.6);
\draw [line width=2pt] (4,3.6)-- (2.4,3.6);
\draw [line width=2pt] (2.4,3.6)-- (2.4,2);
\draw [line width=2pt] (1.6,3.6)-- (1.5982627741675641,2);
\draw [line width=2pt] (0,3.6)-- (1.6,3.6);
\draw [line width=2pt] (1.6,4.4)-- (1.6,6);
\draw [line width=2pt] (1.6,6)-- (2.4,6);
\draw [line width=2pt] (2.4,6)-- (2.4,4.4);
\draw [line width=2pt] (2.4,4.4)-- (4,4.4);
\draw [line width=2pt] (4,4.4)-- (4,6);
\draw [line width=2pt] (1.6,4.4)-- (0,4.4);
\draw [line width=2pt] (0,4.4)-- (0,3.6);
\draw [line width=2pt] (6,2)-- (6,3.6);
\draw [line width=2pt] (6,3.6)-- (7.6,3.6);
\draw [line width=2pt] (7.6,3.6)-- (7.6,2);
\draw [line width=2pt] (8.4,2)-- (8.4,3.6);
\draw [line width=2pt] (8.4,3.6)-- (10,3.6);
\draw [line width=2pt] (10,3.6)-- (10,4.4);
\draw [line width=2pt] (10,4.4)-- (8.4,4.4);
\draw [line width=2pt] (6,6)-- (6,4.4);
\draw [line width=2pt] (6,4.4)-- (7.6,4.4);
\draw [line width=2pt] (7.6,4.4)-- (7.6,6);
\draw [line width=2pt] (7.6,6)-- (8.4,6);
\draw [line width=2pt] (8.4,6)-- (8.4,4.4);
\draw [line width=2pt] (4,0)-- (4,-1.6);
\draw [line width=2pt] (4,-1.6)-- (2.4,-1.6);
\draw [line width=2pt] (4,-2.4)-- (2.4,-2.4);
\draw [line width=2pt] (4,-2.4)-- (4,-4);
\draw [line width=2pt] (2.4,-2.4)-- (2.4,-4);
\draw [line width=2pt] (2.4,-4)-- (1.6,-4);
\draw [line width=2pt] (1.6,-4)-- (1.6,-2.4);
\draw [line width=2pt] (1.6,-2.4)-- (0,-2.4);
\draw [line width=2pt] (0,-2.4)-- (0,-1.6);
\draw [line width=2pt] (0,-1.6)-- (1.6,-1.6);
\draw [line width=2pt] (1.6,-1.6)-- (1.6,0);
\draw [line width=2pt] (2.4,0)-- (2.4,-1.6);
\draw [line width=2pt] (6,-4)-- (6,-2.4);
\draw [line width=2pt] (6,-2.4)-- (7.6,-2.4);
\draw [line width=2pt] (7.6,-2.4)-- (7.6,-4);
\draw [line width=2pt] (7.6,-4)-- (8.4,-4);
\draw [line width=2pt] (8.4,-4)-- (8.4,-2.4);
\draw [line width=2pt] (8.4,-2.4)-- (10,-2.4);
\draw [line width=2pt] (10,-2.4)-- (10,-1.6);
\draw [line width=2pt] (10,-1.6)-- (8.4,-1.6);
\draw [line width=2pt] (8.4,-1.6)-- (8.4,0);
\draw [line width=2pt] (7.6,0)-- (7.6,-1.6);
\draw [line width=2pt] (7.6,-1.6)-- (6,-1.6);
\draw [line width=2pt] (6,-1.6)-- (6,0);
\draw [line width=2pt] (-6,0)-- (-6,2);
\draw [line width=2pt] (-10,2)-- (-10,6);
\draw [line width=2pt] (-12,6)-- (-12,2);
\draw [line width=2pt] (-12,0)-- (-12,-4);
\draw [line width=2pt] (-10,-4)-- (-10,0);
\draw [line width=2pt] (-10,0)-- (-6,0);
\draw [line width=2pt] (-6,2)-- (-10,2);
\draw [line width=2pt] (-10,6)-- (-12,6);
\draw [line width=2pt] (-12,2)-- (-16,2);
\draw [line width=2pt] (-16,2)-- (-16,0);
\draw [line width=2pt] (-16,0)-- (-12,0);
\draw [line width=2pt] (-12,-4)-- (-10,-4);
\end{tikzpicture}
\caption{The first two steps of the construction: the sets $U_1$ and $U_2$} 
\label{fig: U12}
\end{figure}

Suppose $U_i$, $E_i$ and $L_i$ are all defined for $i=1, \dots, n$. By induction, $E_n$ is a union of $4^n$ identical closed squares $\{E_{n, k}\}_{1\le k\le 4^n}$ of size 
$\beta_n\times \beta_n$ that are pairwise disjoint except on the boundary of $[0, 1]^2$, where
\begin{equation}\label{def beta}
\beta_n=2^{-n}\left( 1- \sum_{k=1}^{n} 2^{k-1} \alpha^k \right).
\end{equation}
Since $\alpha\in (0, 0.05)$, we have that 
\begin{equation}\label{def beta1}
\beta_n>2^{-n}\frac{1-3\alpha}{1-2\alpha}>2^{-n-1}.
\end{equation} 
Let $U_{n+1, k}$ be the open $(\a^{n+1}, \b_n)$-cross inscribed in $\inter(E_{n, k})$. Note that each $U_{n+1, k}$ touches a unique cross $U_{n, \ell}$ inside 
$U_n\backslash U_{n-1}$ from left or right. We denote 
$U_{n+1, k}'=U_{n+1, k}\cup I_{r, \a^{n+1}, \b_n}$ or $U_{n+1, k}\cup I_{l, \a^{n+1},\b_n}$ respectively, so that $U_n\cup U_{n+1, k}'$ is simply connected. Then we define
$$
\begin{aligned}
U_{n+1}=U_n\bigcup &\left(\bigcup_{k=1}^{4^n} U_{n+1, k}' \right),\quad
E_{n+1}=\overline{\text{int}(E_n)\backslash U_{n+1}},\\
L_{n+1}&=\TT^2\backslash (U_{n+1}\cup E_{n+1}).
\end{aligned}
$$
By construction, 
\begin{itemize}
\item $U_{n+1}$ is a simply connected open subset in $\TT^2$ and 
$U_{n+1}\supset U_n$; 
\item $E_{n+1}$ is a union of $4^{n+1}$ closed squares of size 
$\beta_{n+1}\times \beta_{n+1}$, which are disjoint except on the boundary of 
$[0, 1]^2$, and $E_{n+1}\subset E_n$;
\item $L_{n+1}$ consists of finitely many line segments, and $L_{n+1}\supset L_n$.
\end{itemize}
Now we define the open set $U$, the Cantor set $E$, and the set $L$ by
$$
U=\bigcup_{n\ge0} U_n, \ \ E=\bigcap_{n\ge 0} E_n,\ \ L=\bigcup_{n\ge 1} L_n. 
$$
It is clear that $U$ is simply connected and $L=\TT^2\backslash (U\cup E)$ consists of countably many line segments. Moreover, 
$$
\Leb_{\TT^2}(E)=\lim_{n\to\infty} \Leb_{\TT^2}(E_n)=\lim_{n\to\infty} 4^n\beta_n^2
=\left( \frac{1-3\alpha}{1-2\alpha} \right)^2>0.
$$
That is, the Cantor set $E$ has positive Lebesgue measure.

\medskip
{\bf Step 2.} Our goal now is to construct a $C^\infty$ diffeomorphism 
$\varphi: \DD^2\to U$, which may not be area preserving.

By the Riemann mapping theorem, or more explicitly, by the Schwarz-Christoffel mapping from the unit disk to polygons, there is a $C^\omega$ diffeomorphism 
$\hvp_0: \DD^2\to U_2$ which can be continuously extended to $\p\DD^2$.

For any $n\ge 2$ choose $\hvp_n: U_n\to U_{n+1}$ as in Lemma~\ref{Ldef phi_n} below and define a sequence of $C^\infty$ diffeomorphisms $\varphi_n: \DD^2\to U_n$ by  
\begin{equation}\label{def psi_n}
\varphi_{n}= \hvp_{n-1} \circ \dots \circ \hvp_1\circ \hvp_0, 
\end{equation}
For any $x\in\DD^2$ we then let $\varphi(x)=\lim_{n\to\infty} \varphi_n(x)$. By Lemma~\ref{Ldef phi_n}(3), for any $x\in \oDD^2$ and $n>j>0$,
\[
\begin{aligned}
d(\varphi_j(x),\varphi_n(x))&\le \sum_{i=j}^{n-1} d(\varphi_i(x),\varphi_{i+1}(x))\\
&\le \sum_{i=j}^{n-1} d\big(\varphi_i(x),\hvp_i(\varphi_{i}(x))\big)\le\sum_{i=j}^{n-1} 2\b^i.
\end{aligned}
\]
By \eqref{def beta1}, this implies that the sequence $\{\varphi_{n}\}$ is uniformly Cauchy and hence, $\varphi$ is well defined and continuous on $\oDD^2$.   

To show that $\varphi$ is a $C^\infty$ diffeomorphism it suffices to show that $\varphi$ is a $C^\infty$ local diffeomorphism, as well as that $\varphi$ is a one-to-one map 
(see \cite{GP74}, Section 1.3). The one-to-one property follows immediately from our construction. From Lemma~\ref{Lphiconv} below, for any $x\in \DD^2$ there exists $n=n(x)\ge 1$ such that $\varphi=\varphi_{n}$ in a neighborhood of $x$, which implies that $\varphi$ is a $C^\infty$ local diffeomorphism.

\medskip
{\bf Step 3.} Now we construct a $C^\infty$ local diffeomorphism $\psi$ on $\DD^2$ such that $h:=\varphi\circ \psi$ is area preserving and can be continuously extended to 
$\oDD^2$.

Denote $\mu=\varphi_* m_U$.  Since both $m_{\DD^2}$ and $m_U$ are normalized Lebesgue measures, we have
$$
\int_{\DD^2} m_{\DD^2} =1=\int_{U} m_{U}=\int_{\DD^2} \mu.
$$
We show that there is a $C^\infty$ diffeomorphism $\psi: \DD^2\to \DD^2$ that can be 
continuously extended to $\p\DD^2$ such that $\psi_* \mu=m_{\DD^2}$. 

Set $\mu_1=m_{\oDD^2}$ and for $n> 1$ define $\mu_n$ such that 
\begin{enumerate}
\item[(i)] $\mu_n\in C^\infty(\oDD^2)$ that is the measure $\mu_n$ is absolutely continuous with respect to $m_{\oDD^2}$ with density function of class $C^\infty$;
\item[(ii)] $\mu_n=\mu$ on $\varphi^{-1}(U_{n-1})$;  
\item[(iii)] $\int_{\varphi^{-1}(U_{n,k}')}\mu_n=\int_{\varphi^{-1}(U_{n,k}')}\mu$ 
for each $k=1,\dots, 4^n$.
\end{enumerate}
It is clear that for any $n\ge 1$, $\int_{\oDD^2} \mu_n=\int_{\oDD^2} \mu=1$.

We need the following version of Moser's theorem, see Lemma~1 in \cite{GS79}.

\begin{lemma}\label{LVolume}
Let $\omega$ and $\mu$ be two volume forms on an oriented manifold $M$ and let $K$ be a connected compact set such that the support of $\omega-\mu$ is contained in the interior of $K$ and $\int_K \omega= \int_K \mu$. Then there is a $C^\infty$ diffeomorphism $\hpsi: M\to M$ such that $\hpsi|{(M\setminus K)}=\id_{(M\setminus K)}$ and $\hpsi_* \omega=\mu$.
\end{lemma}
Note that for a fixed $n$, the sets $\varphi^{-1}(\{U_{n,k}'): k=1, \dots, 4^n\}$ are pairwise disjoint. Now for each $n\ge 1$, we can apply Lemma~\ref{LVolume} $4^n$ times with $K=\varphi^{-1}(\overline{U}'_{n,k})$ for $k=1,\dots, 4^n$ to get a $C^\infty$ diffeomorphism $\hpsi_n: \oDD^2\to \oDD^2$ such that $(\hpsi_n)_*\mu_{n+1}=\mu_{n}$ and 
$\hpsi_n|{\varphi^{-1}(U_{n-1)}}=\id$. Then we let 
\[
\psi_n=\hpsi_n\circ\dots\circ\hpsi_1 \quad\text{and}\quad \psi=\lim_{n\to\infty}\psi_n.
\]
The construction gives $\hpsi_n(\varphi^{-1}(U_{n,k}'))=\varphi^{-1}(U_{n,k}')$. By Lemma~\ref{Ldef phi_n}(3), the Lipschitz constant of $\hvp_n^{-1}$ is less than $1$ if $n$ is large enough. Since $\diam U_{n,k}'\le 2\b_n$, we obtain that 
$\diam \hvp^{-1}_n(U_{n,k}')\le 2\b_n$. This implies that $d(x,\hpsi_n (x))\le 2\b_n$ for any $x\in \DD^2$. By the same arguments as for $\varphi_n$, we can get that the sequence $\{\psi_n(x)\}$ is uniformly Cauchy and hence, $\psi: \oDD^2\to \oDD^2$ is well defined and continuous. Applying the same argument as for $\varphi$, we can also get that $\psi: \DD^2\to \DD^2$ is a $C^\infty$ diffeomorphism.

By construction, we know that $(\psi_n)_*\mu_{n+1}=\mu_1=m_{\DD^2}$. Note that by Lemma~\ref{Lphiconv}, $\DD^2=\cup_{n\ge 1}\varphi^{-1}(U_n)$. Hence, for any 
$x\in \DD^2$ there is an $n>0$ and a neighborhood on which $\mu_{n+k}=\mu_n$ for any $k>0$. It follows that $\psi_*\mu=(\psi_n)_*\mu_{n}=m_{\DD^2}$ on the neighborhood and hence, $\psi_*\mu=m_{\DD^2}$ on $\DD^2$.

{\bf Step 4.} Set $h=\varphi\circ\psi$. Clearly $h: \DD^2\to U$ is a $C^\infty$ diffeomorphism and can be continuously extended to the boundary $\p\DD^2$. Also $h_*m_{U}=\psi_*(\varphi_*m_{U})=\psi_*\mu=m_{\DD^2}$. Since 
$h: \oDD^2\to\overline U$ is continuous, the pre-image of any open set is open and hence, $\CN_\ve$ is a neighborhood of $\p\DD^2$.  All the requirements of the proposition are satisfied.
\end{proof}
To complete the proof of Proposition \ref{PropMain} it remained to prove the two technical lemmas that were used in the above construction.
\begin{lemma}\label{Ldef phi_n}
There is a sequence of $C^\infty$ diffeomorphisms $\hvp_n: U_n\to U_{n+1}$, $n\ge 1$, such that the following properties hold:
\begin{enumerate}
\item $\hvp_{n}=\id$ on $U_{n-1}$;
\item $\hvp_n$ can be continuously extended to $\p U_{n}$;
\item $d(x, \hvp_n(x))\le 2\b_n$ for any $x\in \overline{U}_n$;
\item $\hvp_n^{-1}$ is Lipschitz, and on the set 
$U_{n+1}\cap \{\hvp_{n+1}\not=\id\}$ with Lipschitz constant less than $c\gamma_n^{-1}$, where $\gamma_n=\beta_{n}/\alpha^{{n+1}}$ and $c$ is a constant 
independent of $n$. 
\end{enumerate}
\end{lemma}

\begin{proof}[Proof of the lemma] By construction of $\{U_n\}$, the complement of the set $U_{n+1}\backslash U_n$ is a disjoint union of $4^n$ $(\a^{n+1}, \b_n)$-crosses of the form 
$$
U_{n+1, k}'=U_{n+1, k}\cup I_{l}  \text{ or } U_{n+1, k}'=U_{n+1, k}\cup I_{r}.
$$
By attaching an open square $W_{n, k}$ to the left or right edge of 
$\overline{U}_{n+1, k}'$ respectively, we obtain an augmented cross, denoted by 
$U_{n+1, k}^\sharp$, which is similar to the cross
\begin{equation}\label{fdef C}
\CCC_\gamma:=\big([0, 2+2\gamma]\times [0, 1]  \big) \bigcup 
\big( [1+\gamma, 2+\gamma] \times [-\gamma, 1+\gamma]  \big),
\end{equation}
where $\gamma=\gamma_n=\beta_{n}\alpha^{-(n+1)}$. Assume that the similar map is given by $\eta_{n+1, k}: U_{n+1, k}^\sharp\to \CCC_\gamma$, which is a composition of a translation, enlargement given by $x\to \a^{-(n+1)}x$, and possibly a reflection.
Note that we have
\begin{equation}\label{def gamma_n}
\gamma_n=\frac{\beta_{n+1}}{\alpha^{n+1}}>  \frac{2^{-n-2}}{\alpha^{n+1}} 
> \frac{1}{(2\alpha)^n}>10^n.
\end{equation}

Take the map $\sigma_{\gamma_n}$ as in Sublemma~\ref{SLsquare to cross} below. We then define a map $\hvp_n: U_n \to U_{n+1}$ by
\begin{equation}\label{def phi_n}
\hvp_{n}(x)=
\begin{cases}
\eta_{n+1, k}^{-1} \circ \sigma_{\gamma_n}\circ \eta_{n+1, k}, 
     \ & \ \text{if}\ x\in W_{n, k}, k=1, \dots, 2^n; \\
x, \ & \ \text{elsewhere}.
\end{cases}
\end{equation}
It is easy to see that $\hvp_n:{U_n}\to U_{n+1}$ is a $C^\infty$ diffeomorphism that can be extended to $\p U_n$ continuously. Since $\diam U_{n+1, k}^\sharp\le 2\b$, and 
$\hvp_n$ is a diffeomorphism from $W_{n, k}\subset U_{n+1, k}^\sharp$ to 
$U_{n+1, k}^\sharp$, we must have $d(x, \hvp_n(x))\le 2\b$. So $h$ satisfies  Requirements (1)--(3).

Note that the Lipschitz constant is preserved by conjugacy if the latter is given by a composition of isometries, enlargements, and possibly a reflection. Also note that for each $k$, the set 
$U_{n+1, k}^\sharp\cap \{\hvp_{n+1}\not=\id\}$ is contained in 
$\eta_{n+1, k}^{-1}(\CCC_\gamma^\pm)$, where $\CCC_\gamma^\pm$ is defined in Sublemma~\ref{SLsquare to cross}. So Requirement~(4) of this lemma follows from  Requirement~(3) of the sublemma with $\gamma=\gamma_n$.
\end{proof}

\begin{lemma}\label{Lphiconv}
For any $x\in \DD^2$, there exists $n(x)\ge 1$ such that $\hvp_j(x)=\hvp_{n(x)}(x)$ for any $j\ge n(x)$.
\end{lemma}

\begin{proof} Assuming otherwise, for any $n\ge 1$, we have  that
$\hvp_{n}(x)\in W_{n, k_{n}}\cap \{\hvp_{n}\not= \id\})$ for some $1\le k_{n}\le 4^{n}$. Therefore, by \eqref{def gamma_n} and Lemma~\ref{Ldef phi_n}(4),
{\allowdisplaybreaks
\begin{align*}
d(\hvp_1(x), \p U_1)
&\le c\gamma_1^{-1}\dots c\gamma_n^{-1} d(\hvp_{n+1}(x), \p U_{n+1}) \\
&< \prod_{k=1}^n c \cdot 10^{-k}
= c^{n} \cdot 10^{-\frac{n(n+1)}{2}}  \to 0, \ \ \text{as} \ n\to\infty,
\end{align*}
}
which implies that $\hvp_1(x)\in \p U_1$ leading to a contradiction.
\end{proof}

\begin{sublemma}\label{SLsquare to cross}
Let $\CCC_\gamma$ be defined in \eqref{fdef C}, and let
\begin{align*}
& \CCC_\gamma^+ := \left[1+\gamma, 2+\gamma \right] \times 
\left[1+(\gamma -1)/2, \ 1+ \gamma \right], \\
 \ \  
&\CCC_\gamma^- := \left[1+\gamma, 2+\gamma \right] \times 
\left[  - \gamma, \ -(\gamma -1)/2 \right].
\end{align*}
There exists $c>0$ such that for any $\gamma>10$ there is a homeomorphism $\sigma_\gamma: [0, 1]^2\to \CCC_\gamma$ which has the following properties (see Figure~\ref{fig:varphi_gamma}):
\begin{enumerate}
\item $\sigma_\gamma|{(0, 1)^2}$ is a $C^\infty$ diffeomorphism;
\item $\sigma_\gamma=\id$ in a neighborhood of $\Gamma$, where $\Gamma$ is the boundary of the unit square $[0, 1]^2$ without its right edge;
\item $\sigma_\gamma^{-1}$ is Lipschitz and on $\CCC_\gamma^\pm$ 
the Lipschitz constant is less than $c\gamma^{-1}$.
\end{enumerate}
\end{sublemma}

\begin{figure}[hpt]
\centering
\includegraphics[width=0.6\textwidth]{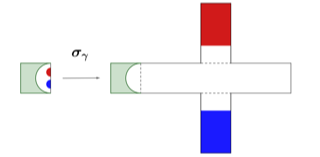}
\caption{The map $\sigma_\gamma$}
\label{fig:varphi_gamma}
\end{figure}

\begin{proof}
Set the function $p(t)=\sqrt{t(1-t)}$ for $t\in [0, 1]$, and consider the domain 
$$
\Omega^\pm:=\{(x_1, x_2)\in \RR^2: \ 0\le x_2\le 1, \ 0\le x_1\le 1\pm p(x_2)\}. 
$$
It is clear that $\Omega^-$ is a neighborhood of $\Gamma$ in $[0, 1]^2$.

We claim that there exists $c_0>0$ such that for any $\kappa>3$,
there exists a homeomorphism 
$\widehat\vvphi=\widehat\vvphi_{\kappa}: [0, 1]^2 \to [0, 2+2\kappa]\times [0, 1]$ such that  
\begin{enumerate}
\item[(H1)] $\left. \widehat\vvphi\right|{(0, 1)^2}$ is a $C^\infty$ diffeomorphism;
\item[(H2)] $ \widehat\vvphi=\mathrm{id}$ on $\Omega^-$;
\item[(H3)] $ \widehat\vvphi^{-1}$ is Lipschitz on the rectangle $\left[1+\kappa, 2+2\kappa \right]\times [0, 1]$, 
with the Lipschitz constant less than $c_0\kappa^{-1}$. 
\end{enumerate}
To see this, we pick a $C^\infty$ non-decreasing function $\chi: \RR\to [0, 1]$ such that 
$\chi(t)=0$ for all $t\le 0$, and $\chi(t)=1$ for all $t\ge 0.1$. Moreover, $\chi$ is sufficiently flat at $t=0$ such that $t\mapsto t^{-k}\chi(t)$ is $C^\infty$ for any $k>0$.
Set $c_1:=\max_{t\in \RR} \chi'(t)$. We shall define 
$\widehat\vvphi=\widehat\vvphi_3\circ \widehat\vvphi_2\circ \widehat\vvphi_1$ 
as follows (See Figure~\ref{fig: hatphi}).

\begin{figure}[hpt]
\centering
\includegraphics[width=0.85\textwidth]{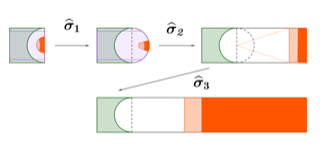}
\caption{Construction of the map $\widehat\vvphi$}
\label{fig: hatphi}
\end{figure}

First, we define a homeomorphism $\widehat\vvphi_1: [0, 1]^2\to \Omega^+$ by 
setting its inverse $\widehat\vvphi_1^{-1}(x_1, x_2)=(x_1- S(x_1, x_2), x_2)$, where 
\begin{align*}
S(x_1, x_2):=
\begin{cases}
p(x_2) \chi\left(\frac{x_1-1+p(x_2)}{2p(x_2)}\right), \ & \ \text{if} \ x_2\ne 0, 1, \\
0, \ & \ \text{if} \ x_2= 0, 1.
\end{cases}
\end{align*}
It is clear that $\widehat\vvphi_1$ is a $C^\infty$ diffeomorphism from $(0, 1)^2$ to the interior of $\Omega^+$ and $\widehat\vvphi_1=\mathrm{id}$ on $\Omega^-$. The Jacobian matrix of $d\widehat\vvphi_1^{-1}$ at $(x_1, x_2)\in \inter(\Omega^+)$ is given by 
$$
d\widehat\vvphi_1^{-1}\left( x_1, x_2 \right)
= \begin{pmatrix} 1- \p_{x_1} S &  - \p_{x_2} S \\ 0 & 1 \end{pmatrix},
$$
where 
$$
\p_{x_1} S= \frac12\chi'\left(\frac{x_1-1+p(x_2)}{2p(x_2)}\right)\in \left[0, \frac{c_1}{2}\right].
$$ 
Furthermore, if $x_2\in [0.1, 0.9]$, then
{\allowdisplaybreaks
\begin{align*}
\left| \p_{x_2} S \right|
=& \left| p'(x_2) \chi \left(\frac{x_1-1+p(x_2)}{2p(x_2)}\right)  
+ \frac{(1-x_1)p'(x_2)}{2p(x_2)} \chi'\left(\frac{x_1-1+p(x_2)}{2p(x_2)}\right) \right| \\
\le & \max_{x_2\in [0.1, 0.9]} \left|p'(x_2)\right| + 
\frac{\max\limits_{x_2\in [0.1, 0.9]} \left|p'(x_2)\right| }{ 2 \min\limits_{x_2\in [0.1, 0.9]} p(x_2)} \cdot c_1 
<2+4c_1.
\end{align*}
}Hence $\widehat\vvphi_1^{-1}$ is Lipschitz on 
$\left\{(x_1, x_2)\in \Omega^+: \ 0.1 \le x_2\le 0.9 \right\}$,
with the Lipschitz constant less than $2+4c_1$.

Second, given any $\kappa>3$, we define a homeomorphism
$\widehat\vvphi_2=\widehat\vvphi_{2, \kappa}: \Omega^+\to [0, 1+\kappa]\times [0, 1]$ by setting its inverse $(z_1, z_2)=\widehat\vvphi_2^{-1}(x_1, x_2)$ such that
$$
\begin{aligned}
\frac{z_1-1}{x_1-1}&=\frac{z_2-\frac12}{x_2-\frac12}\\
&=\frac{t(x_1, x_2)-1+\kappa}{\kappa}\chi \left( t(x_1, x_2) \right)
\frac{\chi\left( r(x_1, x_2) \right)}{r(x_1, x_2)}
\chi\left(\frac{x_1-1}{p(x_2)}\right),
\end{aligned}
$$
where 
$$
t(x_1, x_2):=\frac{(x_1-1)^2+4\kappa^2(x_2-\frac12)^2}{ \kappa^2+ 4(x_1-1)^2(x_2-\frac12)^2}
$$ 
and 
$$
r(x_1, x_2):=\sqrt{(x_1-1)^2+(x_2-\frac12)^2}.
$$ 
Note that we take $\frac{\chi\left( r(1, \frac12) \right)}{r(1, \frac12)}=0$. Also,
for $x_2=0$ or $1$, $\chi\left(\frac{x_1-1}{p(x_2)}\right)=\lim_{t\to-\infty} \chi(t)=0$ if $x_1\le 1$, and  $\chi\left(\frac{x_1-1}{p(x_2)}\right)=\lim_{t\to\infty} \chi(t)=1$ if $x_1>1$.

It is not hard to see that $\widehat\vvphi_2$ is a $C^\infty$ diffeomorphism 
from the interior of $\Omega^+$ to  $(0, 1+\kappa)\times (0, 1)$, and 
$\widehat\vvphi_2=\mathrm{id}$ on $[0, 1]^2$. Moreover, for any 
$(x_1, x_2)\in [1+0.8\kappa, 1+\kappa]\times [0, 1]$, we have that 
$0.3\le t(x_1, x_2)\le 2$ and $1<0.8\kappa\le r(x_1, x_2)\le 1.2\kappa$,
and hence, the inverse $(z_1, z_2)=\widehat\vvphi_2^{-1}(x_1, x_2)$ is given by 
{\allowdisplaybreaks
\begin{equation*}
\frac{z_1-1}{x_1-1}=\frac{z_2-\frac12}{x_2-\frac12}=
\frac{t(x_1, x_2)-1+\kappa}{\kappa r(x_1, x_2)}.
\end{equation*}
}By straightforward calculations, we have 
\begin{equation*}
\left|\p_{x_1} r\right|\le 1, \ \left|\p_{x_2} r\right|\le \kappa^{-1}, \ 
\left|\p_{x_1} t\right|\le 6\kappa^{-1}, \ \text{and} \ 
\left|\p_{x_2} t\right|\le 12,
\end{equation*}
which yields that $\frac{\p z_i}{\p x_j}\le 200\kappa^{-1}$ for $i=1, 2$ and $j=1, 2$,
and thus $\widehat\vvphi_2^{-1}$ is Lipschitz on $[1+0.8\kappa, 1+\kappa]\times [0, 1]$
with the Lipschitz constant less than $200\kappa^{-1}$. Also, since
$$
\left| \frac{z_2-\frac12}{x_2-\frac12} \right|\le \frac{1+\kappa}{0.8\kappa^2}<0.6,
$$
we have that
$$
\widehat\vvphi_2^{-1}\left( [1+0.8\kappa, 1+\kappa]\times [0, 1] \right)
\subset \left\{(x_1, x_2)\in \Omega^+: \ 0.1\le x_2\le 0.9 \right\}.
$$
Finally, we define a $C^\infty$ diffeomorphism 
$$
\widehat\vvphi_3=\widehat\vvphi_{3, \kappa}: [0, 1+\kappa]\times [0, 1]\to [0, 2+2\kappa]\times [0, 1],
$$
given by $\widehat\vvphi_3(x_1, x_2)=(T(x_1), x_2)$, where 
$$
T(x_1)=x_1+\left[10(1+\kappa^{-1})\left(x_1-1-0.8\kappa\right) - x_1\right]
\chi\left(\frac{x_1-1-0.8\kappa}{0.2\kappa} \right).
$$
Note that $\widehat\vvphi_3=\mathrm{id}$ on $[0, 1+0.8\kappa]\times [0, 1]$
and $\widehat\vvphi_3$ maps $[0, 1+0.9\kappa]\times [0, 1]$ onto $[1+\kappa, 2+2\kappa]\times [0, 1]$
with $T(x_1)$ being linear, i.e, $T(x_1)=10(1+\kappa^{-1})\left(x_1-1-0.8\kappa\right)$.
Therefore, $\widehat\vvphi_3^{-1}$ is Lipschitz on the rectangle $\left[1+\kappa, 2+2\kappa \right]\times [0, 1]$, 
with the Lipschitz constant no more than $1$.

Finally, we set $\widehat\vvphi=\widehat\vvphi_3\circ\widehat\vvphi_2\circ \widehat\vvphi_1$. It is easy to see that the function $\widehat\vvphi$ has all the desired properties (H1)-(H3) with the constant $c_0=200(2+2c_1)$.

We now proceed with the proof of Sublemma~\ref{SLsquare to cross}. Set $c=c_0^2$.
For any $\gamma>10$, the above claim yields a homeomorphism 
$\widehat\vvphi=\widehat\vvphi_\gamma$ from $[0, 1]^2$ onto 
$[0, 2+2\gamma]\times [0, 1]$ having Properties (H1)-(H3).
We then attach two rectangular wings to the rectangle $[0, 2+2\gamma]\times [0, 1]$ as follows, see Figure~\ref{fig:varphi_wing}.

\begin{figure}[ht]
 \centering 
  \includegraphics[width=\textwidth]{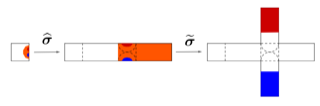} 
  \caption{Attaching rectangular wings by $\widetilde\vvphi$ }
\label{fig:varphi_wing}  
\end{figure}

Take another homeomorphism $\widehat\vvphi'=\widehat\vvphi_{\frac{\gamma-1}{2}}$
from $[0, 1]^2$ onto $[0, 1+\gamma]\times [0, 1]$ having Properties (H1)-(H3). Note that there is a unique planar isometry $\eta_\pm: \RR^2\to \RR^2$ which maps 
$[\gamma, 1+\gamma]\times [0, 1]$ onto $[0, 1]^2$ such that 
$\eta_+(\gamma, 0)=(0, 1)$ and $\eta_+(\gamma, 1)=(1, 1)$, while 
$\eta_-(\gamma, 0)=(1, 0)$ and $\eta_-(\gamma, 1)=(0, 0)$. We then define two homeomorphisms 
$\widetilde\vvphi_+: [\gamma, 1+\gamma ]\times [0, 1] \to [\gamma, 1+\gamma ]\times [0, 1+\gamma]$ and 
$\widetilde\vvphi_-: [\gamma, 1+\gamma ]\times [0, 1] \to [\gamma, 1+\gamma ]\times [-\gamma, 1]$,
which are given by 
$\widetilde\vvphi_\pm=\eta_\pm^{-1}\circ \widehat\vvphi' \circ \eta_\pm$. Further, we define a homeomorphism 
$\widetilde\vvphi: [0, 2+2\gamma ]\times [0, 1] \to \mathcal{C}_\gamma$ by
$$
\widetilde\vvphi(x)=
\begin{cases}
\widetilde\vvphi_+\circ \widetilde\vvphi_-(x), \ & \ x\in  [\gamma, 1+\gamma ]\times [0, 1], \\
x, \ & \ \text{elsewhere}.
\end{cases}
$$
Finally, we take $\vvphi_\gamma=\widetilde\vvphi \circ \widehat\vvphi$, which obviously satisfies Statements (1) and (2) of Lemma~\ref{SLsquare to cross}. It remains to show Statement (3) holds for $\vvphi_\gamma$. Note that 
$$
\widetilde\vvphi^{-1}|{\CCC_\gamma^\pm}
=\widetilde\vvphi_\pm^{-1}|{\CCC_\gamma^\pm}
=\eta_\pm^{-1} \circ \left(\widehat\vvphi'\right)^{-1}|{[(1+\gamma)/2, 1+\gamma]\times [0, 1]} \circ \eta_\pm
$$
is Lipschitz with the Lipschitz constant less than 
$c_0\left(\frac{\gamma-1}{2}\right)^{-1}<c_0$. Moreover, 
$\widetilde\vvphi^{-1}(\CCC_\gamma^\pm)$ is a subset of 
$[\gamma, 1+\gamma ]\times [0, 1]$, on which $\widehat\vvphi^{-1}$ is  Lipschitz with the Lipschitz constant less than $c_0\gamma^{-1}$. Therefore, 
$\vvphi_\gamma^{-1}=\widehat\vvphi^{-1} \circ \widetilde\vvphi^{-1}$ is Lipschitz on 
$\CCC_\gamma^\pm$ with the Lipschitz constant less than 
$c_0^2\gamma^{-1}=c\gamma^{-1}$. The proof of Sublemma~\ref{SLsquare to cross} is complete. 
\end{proof}

\section{Proof of Theorems C and B}\label{SThm C and B}

An important ingredient of our proof of Theorem C is the Katok map 
$g: \overline{\DD^2}\to\overline{\DD^2}$ constructed in \cite{Katok79}. We summarize its properties in the following statement.

\begin{proposition}\label{PropKMap}
There is a $C^\infty$ area preserving diffeomorphism $g:\DD^2\to\bar{ \DD^2}$ which has the following properties:
\begin{enumerate}
\item $g$ is ergodic and in fact, is isomorphic to a Bernoulli map;
\item $g$ has non-zero Lyapunov exponents almost everywhere;
\item there is a neighborhood $\CN$ of $\p \DD^2$ and a smooth vector field 
$Z$ on $\CN$ such that $g$ is the time-$1$ map of the flow generated by $Z$;
\item the map $g$ can be constructed to be arbitrarily flat near the boundary of the disk; more precisely, given any sequence of positive numbers $\rho_n\to 0$ and any sequence of decreasing neighborhoods $\CN_n$ of $\p \DD^2$ satisfying 
\begin{equation}\label{slow down}
\CN_{n}\subset\overline\CN_{n}\subset\CN_{n+1} \text{ and }
\bigcap_{n\ge 1}\CN_n= \p \DD^2,
\end{equation}
one can construct a $C^\infty$ area preserving diffeomorphism $g$ of $\DD^2$ which has Properties (1)-(3) of the proposition and such that  
\begin{equation}\label{fThmC2}
\CN_{n-1}\subset g(\CN_n)\subset \CN_{n+1}
\ \ \text{and} \ \  
\|g - \id \|_{C^n(\CN_n)} \le \rho_n.
\end{equation}
\end{enumerate}
\end{proposition}

\begin{proof}[Proof of Theorem C]
Let $h: \DD^2\to \TT^2$ be the map and $U$ the open dense set both constructed in Proposition~\ref{PropMain}.

Set $V_n=\{x\in U: \dist(x, \p U)<\frac{1}{n}\}$ and $\CN_n=h^{-1}(V_n)$. By Proposition~\ref{PropMain}, $\CN_n$ is a neighborhood of $\p\DD^2$.  

Given a $C^\infty$ map $g:\DD^2\to\DD^2$, choose a number $K_n>0$ such that
\begin{equation}\label{fThmC1}
\begin{aligned}
\| h\circ &\left(g- \id\right)\circ h^{-1}\|_{C^n(V_n\backslash V_{n+1})} \\
\le & K_n \|h\|_{C^n(V_{n-1}\backslash V_{n+2})} 
\|g - \id \|_{C^n(\CN_n)} \|h^{-1} \|_{C^n(V_n\backslash V_{n+1})}.
\end{aligned}
\end{equation}
Fix $\lambda\in (0, 1)$ and let 
\[
\rho_n=\lambda^n/\left( K_n \|h\|_{C^n(V_{n-1}\backslash V_{n+2})} \|h^{-1} \|_{C^n(V_n\backslash V_{n+1})} \right).
\]
Let $g$ be the $C^\infty$ map constructed in Proposition \ref{PropKMap}. Using  \eqref{fThmC2}, we obtain that the map $f: \TT^2\to \TT^2$ given by 
$$
f(x)=
\begin{cases}
(h\circ g \circ h^{-1})(x),  & x\in U; \\
\id, & \text{elsewhere}.
\end{cases}
$$
is well defined. It follows from \eqref{fThmC1} that $\|f-\id\|_{C^n(V_n)}\le \lambda^n$. This implies that $f$ is $C^\infty$-tangent to $\id$ near $\p U$. It is obvious that $f$ satisfies all other requirements of Theorem~C.
\end{proof}

To prove Theorem B we also need a result from \cite{HPT04} (see Proposition 4) showing that there is a smooth isotopy connecting the identity map and the Katok map.  

\begin{proposition}\label{PropKIso}
Let $g: \DD^2\to\DD^2$ be a map given in Proposition \ref{PropKMap}. Then there is a $C^\infty$ map $G: \oDD^2\times [0,1]\to\oDD^2$ such that
\begin{enumerate}
\item for any $t\in [0,1]$ the map $g_t=G(\cdot, t): \oDD^2\to\oDD^2$ is an area-preserving diffeomorphism;
\item $g_0=\id$ and $g_1=g$;
\item $d^n G(x,1)=d^n G(g(x),0)$ for any $n\ge 0$;
\item in a neighborhood $\CN$ of $\p \DD^2$, $g_t|\CN$ is the flow generated by $Z$;
\item the maps $g_t$ are arbitrarily flat near $\p\DD^2$; more precisely, given any sequence of positive numbers $\rho_n\to 0$ and any sequence of decreasing neighborhoods $\CN_n$ of $\p \DD^2$ satisfying \eqref{slow down}, one has that for any $t\in[0,1]$ the map $g=g_t$ satisfies \eqref{fThmC2}. 
\end{enumerate}
\end{proposition}
To prove Theorem~B we start with the smooth isotopy $G(x,t)$ from the above proposition and use the conjugacy map $h$ from Proposition~\ref{PropMain} to get a smooth isotopy $F(x,t)$ on $\TT^2$ that connects the identity map with the map $f_1$ constructed in Theorem~C. We then use this isotopy to define a flow on $\TT^3$ that exhibits the essential coexistence phenomena. Finitely, we make a time change in this flow to obtain a new flow which is ergodic and in fact, is a Bernoulli flow.

\begin{proof}[Proof of Theorem B]
Let $G: \oDD^2\times[0,1]\to \oDD^2$ be the smooth isotopy constructed 
in Proposition~\ref{PropKIso}, $g_t=G(\cdot, t)$, and $h: \DD^2\to U$ be the diffeomorphism constructed in Proposition~\ref{PropMain}.

Choose $V_n$, $\CN_n$ and $\rho_n$ in the same way as in the proof of Theorem~C, and choose $g_t$ such that for all $n\ge 1$ and $t\in [0,1]$ the map $g=g_t$ satisfies \eqref{fThmC2}.

Then we define $F:\TT^3\to \TT^2$ by
$$
F(x, t)=
\begin{cases}
h\circ G(h^{-1}(x), t),  & x\in U, \\
\id, & \text{elsewhere}.
\end{cases}
$$
and denote 
\begin{equation}\label{f-t}
f_t=F(\cdot, t)=h\circ g_t\circ h^{-1}.
\end{equation} 
Similarly, we have 
$$
\sup_{t\in [0, 1]}\| F(\cdot, t)-\id \|_{C^n(V_n)}\le \lambda^n,
$$
which implies that $f_t$ is $C^\infty$-tangent to $\id$ near $\p U$ for each $t\in [0,1]$. Also, by Proposition~\ref{PropKIso}(3), we have 
$d^n F(x,1)=d^n F(f_1(x),0)$ for any $n\ge 0$. In particular, 
\begin{equation}\label{fPropertyF}
F(x,1)=F(f_1(x),0). 
\end{equation}
Next, we use $F$ to define a map on $\TT^3$ and then define a vector field $\hX$.

Let $\CK$ be the suspension manifold over $f_1$, i.e.,
\begin{equation*}
\CK=\{(x, \theta)\in \TT^2\times [0,1]:\ (x,1)\sim (f_1(x),0)\}.
\end{equation*}
The suspension flow $\hf^t: \CK\to \CK$ is generated by the vertical vector field 
$\frac{\p}{\p \theta}$. Note that the restriction of $\hf^t$ on the set 
$\{(x,0)\in U\times [0,1]: (x,1)\sim (f_1(x),0)\}$ has non-zero Lyapunov exponents 
almost everywhere (except along the flow direction), while $\hf^t|{\p U\times \TT}=\id$ has all zero Lyapunov exponents.

We view the $3$-torus as
\begin{equation*}
\TT^3=\{(x, \theta)\in \TT^2\times [0,1]:\ (x,1)\sim (x,0)\}.
\end{equation*}
The map $\hF: \CK\to \TT^3$, given by $\hF(x, \theta)=(F(x, \theta), \theta)$, is well defined, since by \eqref{fPropertyF}, 
$$
\begin{aligned}
\hF(x, 1)&=(F(x, 1), 1)=(F(f_1(x), 0), 1)\\
&=(F(f_1(x), 0), 0)=\hF(f_1(x), 0).
\end{aligned}
$$
Furthermore, $\hF$ is a $C^\infty$ diffeomorphism and $\hF|{\p U\times \TT}=\id$. We now define a $C^\infty$ vector field on $\TT^3$ by $\hX=\hF_*\frac{\p}{\p\theta}$, that is, 
\begin{equation*}
\hX(x, \theta)=:\left(X_1(x_1, x_2,\theta), X_2(x_1, x_2,\theta),1\right).
\end{equation*}
It is clear that $\hX|{\p U\times \TT}=\frac{\p}{\p\theta}$. Since for each $t\in [0,1]$ the map \eqref{f-t} is area reserving, it is easy to see that $\hX$ is divergence-free along 
$\TT^2$-direction, i.e., for any fixed $\theta\in \TT$,
\begin{equation}\label{divergence free}
\dfrac{\p}{\p x_1} X_1(x_1,x_2,\theta)+\dfrac{\p}{\p x_2} X_2(x_1,x_2,\theta)=0.
\end{equation}
Finally, we define a $C^\infty$ vector field $X$ on $\TT^3$ by modifying the last component of $\hX$ as follows. For any sufficiently small $\ve>0$, we let 
\begin{equation}\label{def X tau}
X(x, \theta):=\left(X_1(x_1, x_2,\theta), X_2(x_1, x_2,\theta), \tau(x_1, x_2) \right),
\end{equation}
where $\tau: \TT^2\to [1, 1+\ve]$ is a $C^\infty$ function such that 
\begin{enumerate}
\item[(A1)]
$\|\tau-1\|_{C^1}\le \ve$;
\item[(A2)] 
$\tau$ is a constant greater that $1$ on the closure of some open subset of $\TT^2$;
\item[(A3)]
$\tau=1$ on a neighborhood of $\p U$.
\end{enumerate}
By \eqref{divergence free}, the vector field $X$ is divergence-free along $\TT^2$-direction.
It is also easy to see that $X=\frac{\p}{\p \theta}$ on $\p U\times \TT$.

We denote by $f^t: \TT^3\to \TT^3$ the flow generated by $X$.
Following the arguments in \cite{HPT04}, it is not hard to show that 
$f^t|{U\times \TT}$ is a Bernoulli flow with non-zero Lyapunov exponents 
almost everywhere (except along the flow direction), 
while $f^t|{\p U\times \TT}=\id$ has all zero Lyapunov exponents.
\end{proof}

\section{The Hamiltonian flow: Proof of Theorem~A}\label{SThm A}

We shall prove that the vector field $X=(X_1, X_2, \tau)$ obtained in
Theorem~B can be embedded as a Hamiltonian vector field in the 
$4$-dimensional manifold $\CM=\TT^3\times \RR$.

Let $X=(X_1, X_2, \tau)$ be the vector field given in Theorem~B that
generates the flow $f^t$ in $\TT^3$.

Let $\Theta=\Theta(x_1, x_2, \theta)$ be the backward hitting time to the zero level for the 
flow $f^t$ initiated at $(x_1, x_2, \theta)\in \TT^3$, that is, 
there is a unique point $(\widehat{x}_1, \widehat{x}_2)$ and a unique value 
$\Theta\in [0, 1)\cong \TT$ such that $f^\Theta(\widehat{x}_1, \widehat{x}_2, 0)=(x_1, x_2, \theta)$. 
It is easy to see that the function $\Theta$ is $C^\infty$ smooth on $\TT^3$ and satisfies
\begin{equation}\label{speed change}
X_1 \frac{\p \Theta}{\p x_1} + X_2 \frac{\p \Theta}{\p x_2} + \tau \frac{\p \Theta}{\p \theta} =1.
\end{equation}
Indeed, let $\gamma(t)=f^t(\widehat{x}_1, \widehat{x}_2, 0)$ for $t\in [0, \Theta]$, we have 
$\Theta(\gamma(t))=t$. Taking derivative $\frac{d}{dt}$ on both sides and then letting $t=\Theta$, we obtain 
\eqref{speed change}.

We now consider the 4-dimensional manifold $\CM=\TT^3\times \RR$,
endowed with the standard symplectic form $\omega=dx_1\wedge dx_2 + d\theta\wedge dI$.
The diffeomorphism $\Phi: \CM\to \CM$ given by $\Phi(x_1, x_2, \theta, I)=(x_1, x_2, \Theta, I)$
pulls $\omega$ back to a closed 2-form:
\begin{equation}\label{non-standard symp}
\begin{split}
\homega=&\Phi^*\omega=dx_1\wedge dx_2 + d\Theta\wedge dI \\
=& dx_1\wedge dx_2 + \frac{\p \Theta}{\p x_1}  dx_1 \wedge dI
+\frac{\p \Theta}{\p x_2}  dx_2 \wedge dI +\frac{\p \Theta}{\p \theta}  d\theta \wedge dI.
\end{split}
\end{equation}
We may further assume that $\homega$ is non-degenerate, 
since $\|\homega-\omega\|_{C^0}\le \|\Theta -\theta\|_{C^1}$ could be sufficiently small if the function $\tau$ 
in \eqref{def X tau} is chosen such that $\|\tau-1\|_{C^1}$ is sufficiently small. 
Therefore, $\homega$ is also a symplectic form on $\CM$.

Let $\wH=\wH(x_1, x_2, \theta): \TT^3\to \RR$ be a function which for each 
$\theta\in\TT^2$ is a solution of the following system of equations
\begin{equation}\label{Hamiltonian PDE}
\begin{cases} 
\dfrac{\p \wH}{\p x_2}=X_1, \ -\dfrac{\p \wH}{\p x_1}=X_2, & \ \text{on} \ U, \\
\\
\wH=0 \ & \ \text{on}\ \p U,
\end{cases}
\end{equation}
whose existence is proved in Lemma~\ref{LtildeH} below.
Define $\hH: \CM\to \RR$ by 
\begin{equation}\label{fdef Hamiltonian}
\hH(x_1, x_2, \theta, I)=\wH(x_1, x_2, \Theta(x_1, x_2, \theta)) + I.
\end{equation}
Let $X_{\hH}=(X_1, X_2, \tau, v)$ where 
$$
v(x_1, x_2, \theta, I)
:=\frac{\p \wH}{\p \theta}\left(x_1, x_2, \Theta(x_1, x_2, \theta) \right).
$$
By Lemma~\ref{LhatH} below, $X_{\hH}$ is a Hamiltonian vector field on $\CM$ for the Hamiltonian function $\hH$ with respect to the non-standard symplectic form $\homega$.

Given any $e\in \RR$, the corresponding energy surface is given by
$$
\widehat\CM_e=\{\hH=e\}=\left\{(x_1, x_2, \theta, I)\in \CM:\ I=e-\wH(x_1, x_2, \Theta(x_1, x_2, \theta))\right\}.
$$
Define a $C^\infty$ diffeomorphism $\widehat\Psi_e: \TT^3\to \widehat\CM_e$ by 
$$
\widehat\Psi_e(x_1, x_2, \theta)=(x_1, x_2, \theta, e-\wH(x_1, x_2, \Theta(x_1, x_2, \theta))).
$$
It follows from \eqref{speed change} that $\left(\widehat\Psi_e\right)_*X=X_{\hH}$, and hence, the Hamiltonian flow restricted on $\widehat\CM_e$ (under the non-standard symplectic form $\homega$) is $\widehat\Psi_e\circ f^t\circ \widehat\Psi_e^{-1}$. 

Finally, setting $H=\Phi_*\hH=\hH\circ \Phi^{-1}$ and $X_H=\Phi_*X_{\hH}$, we have that
\begin{equation*}
\omega(X_H, \cdot)=\Phi_*\homega(\Phi_*X_{\hH}, \cdot)
=\Phi_*\left[ \homega(X_{\hH}, \cdot) \right]=\Phi_*d\hH=d\Phi_*\hH=dH.
\end{equation*}
This means that $X_H$ is a Hamiltonian vector field on $\CM$ for the Hamiltonian function $H$ with respect to the standard symplectic form $\omega$. For any $e\in \RR$, the corresponding energy surface is given by
$$
\CM_e=\{H=e\}=\Phi\{\hH=e\}=\Phi\widehat\CM_e.
$$
Furthermore, we define $\Psi_e=\Phi\circ \widehat\Psi_e: \TT^3\to \CM_e$, then
it is clear that $\left(\Psi_e\right)_*X=X_{H}$, and hence, 
the Hamiltonian flow restricted to $\CM_e$ (under the symplectic form $\omega$) is $\Psi_e\circ f^t\circ \Psi_e^{-1}$. This completes the proof of Theorem A subject to the two technical result that we now present.

\begin{lemma}\label{LtildeH}
The system \eqref{Hamiltonian PDE} has a solution $\wH: \TT^3\to \RR$.
\end{lemma}
\begin{proof}[Proof of the lemma]
Recall that 
$$
X(x_1, x_2, \theta)=\left( X_1(x_1, x_2, \theta), X_2(x_1, x_2, \theta), \tau(x_1, x_2) \right)
$$
is a $C^\infty$ vector field on $\TT^3$ such that $X|{\partial U\times \TT}=\frac{\partial}{\partial \theta}$, which means that
$$
X_1(x_1, x_2, \theta)=X_2(x_1, x_2, \theta)=0 \ \ \text{for any} \ (x_1, x_2)\in \partial U \ \text{and} \ \theta\in \TT.
$$
Moreover, $X$ is divergence-free along $\TT^2$-direction, i.e., \eqref{divergence free} holds for any $\theta\in \TT$.

Now we extend the domains of $X_1$ and $X_2$ onto $\RR^3$ by periodicity, that is, we set
$$
X_i(\widetilde{x}_1, \widetilde{x}_2, \widetilde{\theta}) = X_i(x_1, x_2, \theta), \ \ i=1, 2,
$$
for any $(\widetilde{x}_1, \widetilde{x}_2, \widetilde{\theta})\in \RR^3$ such that 
$(\widetilde{x}_1, \widetilde{x}_2, \widetilde{\theta} ) \equiv (x_1, x_2, \theta) \pmod {\ZZ^3}$, where $(x_1, x_2, \theta)\in \TT^3$. Further, we consider a family of smooth $1$-forms on $\RR^2$ given by
\begin{equation*}\label{def omega}
\omega_{\widetilde{\theta}} \left(\widetilde{x}_1, \widetilde{x}_2 \right) = 
X_2(\widetilde{x}_1, \widetilde{x}_2, \widetilde{\theta})\ d\widetilde{x}_1  
-X_1(\widetilde{x}_1, \widetilde{x}_2, \widetilde{\theta})\ d\widetilde{x}_2,
\end{equation*}
in which we view $\widetilde{\theta}$ as a parameter.
It is clear that $\omega_{\widetilde{\theta}} \left(\widetilde{x}_1, \widetilde{x}_2 \right)$
are periodic in both $\widetilde{x}_1$ and $\widetilde{x}_2$ as well as in the parameter $\widetilde{\theta}$.
Moreover, $\omega_{\widetilde{\theta}} \left(\widetilde{x}_1, \widetilde{x}_2 \right)=0$
if either $\widetilde{x}_1\in \ZZ$ or $\widetilde{x}_2\in \ZZ$.

For each fixed $\widetilde{\theta}\in \RR$, Equation \eqref{divergence free}  implies that the $1$-form 
$\omega_{\widetilde{\theta}}$ is closed, i.e., $d\omega_{\widetilde{\theta}}=0$.
By Poincar\'e Lemma, the form 
$\omega_{\widetilde{\theta}}$ is exact since $\RR^2$ is contractible (see Chapter 8 in \cite{Mun91}). More precisely, $\omega_{\widetilde{\theta}}=du_{\widetilde{\theta}}$, where we choose a particular potential function $u_{\widetilde{\theta}}$ by 
\begin{equation*}\label{def u}
u_{\widetilde{\theta}}(\widetilde{x}_1, \widetilde{x}_2) = \int_{\Gamma} \omega_{\widetilde{\theta}} \ ,
\end{equation*}
for any smooth path $\Gamma$ from $(0, 0)$ to $(x_1, x_2)$ in $\RR^2$ (note that the path integral is independent of the choice of $\Gamma$). Moreover, we claim that $u_{\widetilde{\theta}}$ is 
periodic in $(\widetilde{x}_1, \widetilde{x}_2)$, that is, 
$$
u_{\widetilde{\theta}}(\widetilde{x}_1', \widetilde{x}_2')=u_{\widetilde{\theta}}(\widetilde{x}_1, \widetilde{x}_2)\ \
\text{if} \ (\widetilde{x}_1',\widetilde{x}_2')\equiv (\widetilde{x}_1,\widetilde{x}_2)\pmod
{\ZZ^2}.
$$ 
Indeed, by periodicity of $\omega_{\widetilde{\theta}}$, it suffices to show that 
$\int_{\Gamma} \omega_{\widetilde{\theta}}=0$
for two special types of paths, i.e., 
$$
\Gamma=\Gamma^1_a:=\left\{(\widetilde{x}_1, a): \ 0\le \widetilde{x}_1\le 1 \right\} \ \text{and} \ \
\Gamma=\Gamma^2_a:=\left\{(a, \widetilde{x}_2): \ 0\le \widetilde{x}_2\le 1 \right\},
$$
for any $a\in \RR$. When $\Gamma=\Gamma^1_a$ we denote the bounded region 
$$
\Omega^1_a:=\left\{(\widetilde{x}_1, \widetilde{x}_2):\ 0\le \widetilde{x}_1\le 1, \ 0\le \widetilde{x}_2\le a \right\},
$$
and note that $\omega_{\widetilde{\theta}}$ vanishes on 
$\partial \Omega^1_a\backslash \Gamma^1_a$. By Stokes' Theorem and the fact that 
$\omega_{\widetilde{\theta}}$ is a closed form, we have
$$
\int_{\Gamma^1_a} \omega_{\widetilde{\theta}}=\int_{\partial \Omega^1_a} \omega_{\widetilde{\theta}}=
\int_{\Omega^1_a} d\omega_{\widetilde{\theta}}=0.
$$
In a similar fashion, we can show that 
$\int_{\Gamma^2_a} \omega_{\widetilde{\theta}}=0$ as well.
Thus, $u_{\widetilde{\theta}}$ is periodic in $(\widetilde{x}_1, \widetilde{x}_2)$.

We further notice that if $\widetilde{\theta}'\equiv \widetilde{\theta} \pmod \ZZ$, 
then $\omega_{\widetilde{\theta}'}=\omega_{\widetilde{\theta}}$ and hence, $u_{\widetilde{\theta}'}=u_{\widetilde{\theta}}$. To summarize, 
$u_{\widetilde{\theta}}(\widetilde{x}_1, \widetilde{x}_2)$ is periodic in all arguments, that is,
$$
u_{\widetilde{\theta}'}(\widetilde{x}_1', \widetilde{x}_2')=u_{\widetilde{\theta}}(\widetilde{x}_1, \widetilde{x}_2)
\ \ \text{if} \ (\widetilde{x}_1', \widetilde{x}_2', \widetilde{\theta}')\equiv (\widetilde{x}_1, \widetilde{x}_2, \widetilde{\theta}) \pmod {\ZZ^3}.
$$ 
Therefore, the function $\wH: \TT^3\to \RR$ given by 
$$
\wH\left(x_1, x_2, \theta \right)=u_{\widetilde{\theta}}(\widetilde{x}_1, \widetilde{x}_2),
$$
where $(x_1, x_2, \theta)\equiv (\widetilde{x}_1, \widetilde{x}_2, \widetilde{\theta} ) \pmod {\ZZ^3}$, is well-defined.
It follows from the definition of $\omega_{\widetilde{\theta}}$ and $u_{\widetilde{\theta}}$ 
that $\wH$ is a solution of \eqref{Hamiltonian PDE}.
\end{proof}
\begin{lemma}\label{LhatH}
The Hamiltonian vector field in $\CM$ corresponding to the non-standard symplectic form 
$\homega$ and the Hamiltonian function 
$$
\hH(x_1, x_2, \theta, I)=\wH(x_1, x_2, \Theta(x_1, x_2, \theta)) + I.
$$
is given by $X_{\hH}:=(X_1, X_2, \tau, v)$, where 
$$
v(x_1, x_2, x_3, I):=\frac{\p \wH}{\p \theta}\left(x_1, x_2, \Theta(x_1, x_2, \theta) \right).
$$
\end{lemma}
\begin{proof}[Proof of the lemma] Using \eqref{speed change}, it is straightforward to show that
{\allowdisplaybreaks
\begin{align*}
\homega(X_{\hH}, \cdot) &= \left(-X_2 - v \frac{\p \Theta}{\p x_1}\right) dx_1 
                             +\left(X_1 - v \frac{\p \Theta}{\p x_2}\right) dx_2 - v \frac{\p \Theta}{\p\theta} d\theta \\
                             & \ + \left(X_1 \frac{\p \Theta}{\p x_1} + X_2 \frac{\p \Theta}{\p x_2} + \tau \frac{\p \Theta}{\p \theta} \right) dI \\
                             &=\frac{\p \hH}{\p x_1}  dx_1 + \frac{\p \hH}{\p x_2}  dx_2 + \frac{\p \hH}{\p \theta}  d\theta + dI=d\hH.
\end{align*}
}
Thus, $X_{\hH}$ is the Hamiltonian vector field of $\hH$ under the symplectic form 
$\homega$.
\end{proof}

\bibliography{Coexistence}{}
\bibliographystyle{alpha}
\end{document}